\documentclass[10pt,a4paper]{scrartcl}

\usepackage[english]{babel}
\usepackage{abstract}
\usepackage{amsmath}
\usepackage{amsfonts}
\usepackage{amssymb}
\usepackage{suffix}
\usepackage{mathtools}
\usepackage{framed}
\usepackage[svgnames]{xcolor}
\usepackage[amsmath, amsthm, thmmarks, framed]{ntheorem}
\usepackage{comment}
\usepackage{tikz}
\usetikzlibrary{calc, decorations.markings}
\usepackage[textsize=footnotesize]{todonotes}

\definecolor{lightgray}{gray}{1}

\newtheorem{theorem}{Theorem}[section]
\newtheorem{corollary}[theorem]{Corollary}
\newtheorem{proposition}[theorem]{Proposition}
\newtheorem{lemma}[theorem]{Lemma}
\theoremstyle{remark}
\newtheorem{remark}[theorem]{Remark}

\makeatletter
\def\blfootnote{\xdef\@thefnmark{}\@footnotetext}
\makeatother

\DeclarePairedDelimiterX\MeijerM[3]{\lparen}{\rparen}{\begin{matrix}#1 \\ #2\end{matrix}\delimsize\vert\,#3}

\newcommand\MeijerG[7]{ G^{\,#1,#2}_{#3,#4}\MeijerM*{#5}{#6}{#7}}
\newcommand{\HyperG}[5]{{}_{#1}F_{#2}\MeijerM*{#3}{#4}{#5}}

\newcommand{\ds}{\displaystyle}
\DeclareMathOperator{\Tr}{Tr}

\DeclareMathOperator{\Real}{Re}
\DeclareMathOperator{\Imag}{Im}
\renewcommand{\Re}{\Real}
\renewcommand{\Im}{\Imag}
\DeclareMathOperator*{\Res}{Res}

\newcommand{\R}{\mathbb{R}}

\numberwithin{equation}{section}

\title{Singular values of products of random matrices and polynomial ensembles}
\author{Arno B. J. Kuijlaars and Dries Stivigny}
\date{\today}

\begin{document}

\maketitle
\blfootnote{KU Leuven, Department of Mathematics, Celestijnenlaan 200B box 2400, 3001 Leuven, Belgium. 
E-mail: arno.kuijlaars@wis.kuleuven.be, dries.stivigny@wis.kuleuven.be }
%

\begin{abstract}
Akemann, Ipsen, and Kieburg showed recently that the squared singular values of  a product of 
$M$ complex Ginibre matrices are distributed according to a determinantal point process. 
We introduce the notion of a polynomial ensemble and show how their result can be 
interpreted as a transformation of polynomial ensembles. 
We also show that the squared singular values of the product of $M-1$ complex Ginibre matrices 
with one truncated unitary matrix is a polynomial ensemble, and we derive a double integral 
representation for the correlation kernel associated with this ensemble. We use this to 
calculate the scaling limit at the hard edge, which turns out to be the same scaling limit 
as the one found by Kuijlaars and Zhang for the squared singular values of a product of $M$ 
complex Ginibre matrices. 
Our final result is that these limiting kernels also appear as scaling limits for  the 
biorthogonal ensembles of Borodin with parameter $\theta > 0$, in case $\theta$ or $1/\theta$
is an integer. 
This further supports the conjecture that these kernels have a universal character.
\end{abstract}

\section{Introduction}

\subsection{Products of Ginibre matrices}
There is a remarkable recent development in the understanding of the structure of eigenvalues
and singular values of products of complex Ginibre matrices at the finite size level.
Both the eigenvalues and the singular values turn out to have a determinantal structure. 
For the eigenvalues this was shown by Akemann and Burda \cite{Akemann_Burda}. Related results
that involve also products with inverses of complex Ginibre matrices are in 
\cite{Krishnapur, Adhikari_Reddy_Reddy_Saha},
and products with truncated unitary matrices in \cite{Akemann_Burda_Kieburg_Nagao, Ipsen_Kieburg}.

The eigenvalue probability density function in these models takes the form 
\[ \frac{1}{Z_n} \left| \Delta(z) \right|^2 \, \prod_{j=1}^n w(z_j), \qquad  (z_1, \ldots, z_n) \in \mathbb C^n, \]
where $\Delta(z) = \prod_{j<k} (z_k-z_j)$ is the Vandermonde determinant, and $w$ is 
a weight function that is expressed in terms of a Meijer G-function (see e.g.\ \cite{Beals_Szmigielski, Luke} 
or the appendix for an introduction).

The determinantal structure also holds for the squared singular values of products $Y = G_M \cdots G_1$ of
$M \geq 1$ complex Ginibre matrices. Suppose $G_j$ has size $(n + \nu_j) \times (n + \nu_{j-1})$ with $\nu_0 =0$,
$\nu_1, \ldots, \nu_M \geq 0$. Then the joint probability density function takes the form
\begin{equation} \label{eq: jpdf_gin}
	\frac{1}{Z_n} \Delta(y)  \det\left[w_{k-1}(y_j)\right]_{j, k = 1}^n, \qquad (y_1, \ldots, y_n) \in [0,\infty)^n,
\end{equation}
where the $y_j$'s  are the squared singular values of $Y$, and
\begin{equation} \label{wk1}
	w_k(y) = \MeijerG{M}{0}{0}{M}{-}{\nu_M, \ldots, \nu_2, \nu_1 + k}{y}, 
\end{equation}
is again a Meijer G-function. This was shown by Akemann, Kieburg, and Wei \cite{Akemann_Kieburg_Wei}
in the case of square matrices, and by Akemann, Ipsen, and Kieburg \cite{Akemann_Ipsen_Kieburg}
for general rectangular matrices. 

\subsection{Polynomial ensembles}
The density \eqref{eq: jpdf_gin} defines a biorthogonal ensemble which is a special case of a determinantal
process. Because of the Vandermonde determinant $\Delta(y)$ in \eqref{eq: jpdf_gin} there is a connection
with polynomials and we call \eqref{eq: jpdf_gin} a \textbf{polynomial ensemble}.
A general polynomial ensemble is of the form
\begin{equation} \label{eq: pol_ensemble}
	\frac{1}{Z_n} \Delta(x)  \det\left[f_{k-1}(x_j)\right]_{j, k = 1}^n, \qquad (x_1, \ldots, x_n) \in \mathbb R^n,
\end{equation}
for certain functions $f_0, \ldots, f_{n-1}$.
In such an ensemble the correlation kernel is
\[ K_n(x,y) = \sum_{k=0}^{n-1} P_k(x) Q_k(y) \]
where $P_k$ is a monic polynomial of degree $k$ such that
\begin{equation} \label{eq: pol_ensemble_Pk} 
	\int_0^{\infty} P_k(x) f_j(x) dx = 0 \qquad \text{for } j = 0, 1, \ldots, k-1, \text{ and } k = 0, 1, \ldots, n,
	\end{equation}
and $Q_k$ is in the linear span of $f_0, \ldots, f_k$ such that
\[ \int_0^{\infty} x^j Q_k(x) dx = \delta_{j,k} \qquad \text{for } j = 0, 1, \ldots, k,
	\text{ and } k = 0, 1, \ldots, n-1. \]
If $f_k(x) = x^k f_0(x)$ for every $k$, then \eqref{eq: pol_ensemble} is an orthogonal
polynomial ensemble \cite{Konig} and \eqref{eq: pol_ensemble_Pk} reduces
to the conditions for an orthogonal polynomial with respect to $f_0$.
It is also worth noting that in a polynomial ensemble, $P_n$ is the average characteristic polynomial
\[ P_n(x) = \mathbb E \left[ \prod_{j=1}^n (x-x_j) \right] \]
with the expectation taken over \eqref{eq: pol_ensemble}.

The first aim of the present work is to interpret the result of Akemann et al.\ as
a transformation of polynomial ensembles. The result may be stated as follows:
suppose $X$ is a random matrix whose squared singular values form a polynomial
ensemble and that $G$ is a complex Ginibre matrix. Then the squared singular values of $Y = GX$ are also a polynomial
ensemble. See Theorem \ref{thm: distr_sv_prod}  below for a precise formulation. 

We can use the theorem repeatedly, and we obtain that the squared singular
values of $Y = G_{M-1} \cdots G_1 X$ are also a polynomial ensemble, for any
$M \geq 1$ and complex Ginibre matrices $G_1, \ldots, G_{M-1}$.

The theorem applies to any random matrix $X$ whose squared singular values are
a polynomial ensemble. Taking for $X$ a complex Ginibre matrix itself, we obtain the
result of Akemann et al., and by taking for $X$ the inverse of a product of complex Ginibre
matrices, we rederive a recent result of Forrester \cite{Forrester}. In both these
examples, the functions in the polynomial orthogonal ensembles are expressed as
Meijer G-functions.

We  consider one new example
where $X$ is a truncation of a Haar distributed unitary matrix for which it is known that
the squared singular values are a Jacobi ensemble on $[0,1]$. We find explicit
expressions that are once again in terms of Meijer G-functions, see Corollary \ref{cor: prod_gin_trunc}.

\subsection{Scaling limits at the hard edge}

The correlation kernels $K_n$ for the polynomial ensemble \eqref{eq: jpdf_gin}-\eqref{wk1}
have an interesting large $n$ scaling limit at the hard edge
\[ \lim_{n \to\infty} \frac{1}{n} K_n \left(\frac{x}{n}, \frac{y}{n} \right) = K_{\nu_1, \ldots, \nu_M}(x,y) \]
with a limiting kernel that depends on $M$ parameters, $\nu_1, \ldots, \nu_M$,
see \cite{Kuijlaars_Zhang} and Theorem \ref{thm: scaling_KZ} below for a precise statement.

For $M=1$ they reduce to the hard edge Bessel kernels, see e.g.\ \cite{Tracy_Widom} or 
\cite[Section 7.2]{Forrester_LogGases}, and for $M=2$ these kernels already
appeared in work of Bertola et al.\ \cite{Bertola_Gekhtman_Szmigielski} 
on the Cauchy-Laguerre two matrix model. 
In \cite{Forrester} Forrester obtained the same
family of limiting kernels for the squared singular values of a product of complex Ginibre matrices
with the inverse of another product of complex Ginibre matrices.
Differential equations for the gap probabilities are in \cite{Strahov}.

For results on the global distribution of the points in \eqref{eq: jpdf_gin}-\eqref{wk1},
see e.g.\ \cite{Burda_JLNS, Forrester_Liu, Neuschel, Penson_Zyczkowski, Zhang}. 

The second aim of this paper is to provide two more examples of models with the kernels 
$K_{\nu_1, \ldots, \nu_M}$ as scaling limit. We show that the new example 
with the product of Ginibre matrices with one truncated unitary matrix falls into
this category. 
In the second example  we consider the biorthogonal ensembles
\begin{equation} \label{eq:Bor1}  
	\frac{1}{Z_n} \prod_{j < k } (x_k -x_j) \, \prod_{j < k} (x_k^{\theta} - x_j^{\theta}) \, \prod_{j=1}^n w(x_j) 
	\end{equation}
with all $x_j > 0$ and $\theta > 0$. Borodin \cite{Borodin} found the hard edge scaling limits for the cases where
$w(x) = x^{\alpha} e^{-x}$ or $w(x) = x^{\alpha} \chi_{[0,1]}(x)$. The scaling limit depends on
the two para\-meters $\theta > 0$ and $\alpha > -1$.
We show that, after suitable rescaling, the limiting kernels  belong to the class of kernels $K_{\nu_1, \ldots, \nu_M}$
provided that $\theta$ or $1/\theta$ is an integer, see Theorem \ref{thm:theta-ensemble}.

This last example in particular supports the conjecture that the kernels $K_{\nu_1, \ldots, \nu_M}$
have a universal character and that they might appear as scaling limits in other interesting random models  as well.


\section{Transformation of Polynomial Ensembles}

A complex Ginibre matrix $G$ of size $m \times n$ has independent entries whose real and
imaginary parts are independent and have a standard normal distribution. The probability distribution
can be written as 
\begin{equation} \label{eq: Ginibre} 
	\frac{1}{Z_n} e^{- \Tr G^* G} dG 
	\end{equation}
where $dG = \prod_{j=1}^m \prod_{k=1}^n d \Re G_{j,k} \, d \Im G_{j,k}$ and $Z_n$ is a normalization
constant.

The main result of this section is the following:

\begin{theorem} \label{thm: distr_sv_prod}
Let $\nu \geq 0$ and $l \geq n \geq 1$ be integers and let $G$ be a complex Ginibre random matrix of size 
$(n + \nu) \times l$. Let $X$ be a random matrix of size $l \times n$, independent of $G$, such 
that the squared singular values $x_1, \ldots, x_n$ of $X$ have a joint probability density 
function on $[0,\infty)^n$ that is proportional to
\begin{equation} \label{Xpdf}
\Delta(x) \det\left[f_{k-1}(x_j)\right]_{j, k = 1}^n,
\end{equation}
for certain functions $f_0, \ldots, f_{n-1}$.
Then the squared singular values $y_1, \ldots, y_n$ of $Y = GX$ are distributed with a joint probability density function on $[0,\infty)^n$  proportional to
\begin{equation} \label{Ypdf}
\Delta(y) \det\left[g_{k-1}(y_j)\right]_{j, k = 1}^n
\end{equation}
with
\begin{equation} \label{eq: mellin_transf_f}
g_k(y) = \int_0^{\infty} x^{\nu} e^{-x} f_k\left(\frac{y}{x}\right) \frac{dx}{x}, \qquad k = 0, \ldots, n-1.
\end{equation}
\end{theorem} 

The theorem says that left multiplication by a complex Ginibre matrix maps
polynomial ensembles to polynomial ensembles. 
Observe  that $g_k$ is the Mellin convolution \cite[formula 1.14.39]{Olver:2010} 
of $f_k$ with the function $x \mapsto  x^{\nu} e^{-x}$.

Before we prove the theorem, we state an auxiliary result, which is essentially 
contained in \cite[section 2.1]{Baik_Ben-Arous_Peche} and also in \cite{Akemann_Ipsen_Kieburg}. 
For clarity, we give a detailed proof of this result. 

\begin{lemma}  \label{lemma: distr_sv_prod_fixed}
Let $\nu, l, n$ and $G$ be as in Theorem \ref{thm: distr_sv_prod}.
Let $X$ be a non-random matrix of size $l \times n$ with non-zero squared singular values $x_1, \ldots, x_n$.
Then the squared singular values $y_1, \ldots, y_n$ of $Y = GX$ have a joint probability density function
on $[0, \infty)^n$ that is proportional to
\begin{equation} \label{Y0pdf} 
	\frac{\Delta(y)}{\Delta(x)} \det \left[ \frac{y_j^{\nu}}{x_k^{\nu+1}} e^{-y_j/x_k} \right]_{j,k=1}^n
	\end{equation}
where the proportionality constant does not depend on $X$. 
\end{lemma}
In case some of the  $x_k$ are the same, we have to interpret
\eqref{Y0pdf} in the appropriate limiting sense using l'H\^opital's rule.

\begin{proof} 
First we show that we can reduce the proof to the case $l = n$. 
Assume $l > n$. Then any matrix $X$ of size $l \times n$ can be written as 
\[ X = U \begin{pmatrix} X_0 \\ O \end{pmatrix} \] 
where $U$ is an $l \times l$ unitary, $X_0$ is an $n \times n$, and $O$ is a zero matrix 
of size $(l-n) \times n$. Then by the unitary invariance of Ginibre random matrix ensembles, 
we have that $Y = GX$, has the same distribution
of singular values as $Y_0  = G_0 X_0$, where $G_0$ is an $(n + \nu) \times n$ complex Ginibre matrix. 

\bigskip

So in the rest of the proof we assume that $l =n $ and $X$ is a square matrix of size $n \times n$
with squared singular values $x_1, \ldots, x_n$.

Then it is known that the change of variables $G \mapsto Y = GX$, with $X$ being fixed, where
$G$ and $Y$ are $(n +\nu) \times n$ matrices, has  a Jacobian (see e.g.\ \cite[Theorem 3.2]{Mathai})
\begin{equation} \label{JacobianGtoY}
	\det(X^{\ast} X)^{-(n + \nu)} = \prod_{k = 1}^n x_k^{-(n + \nu)}.
\end{equation}
Thus under the mapping $G \mapsto Y$
the Ginibre probability distribution \eqref{eq: Ginibre}  transforms, up to a constant, into 
\begin{equation} \label{eq: G_to_Y}
	e^{-\Tr(G^{\ast}G)} dG =  \left(\prod_{k = 1}^n x_k^{-(n + \nu)} \right) e^{-\Tr(Y^{\ast}Y (X^{\ast}X)^{-1})} dY.
\end{equation}

Next we write $Y = U \Sigma V$ in its singular value decomposition. Thus $\Sigma$ is a diagonal
matrix with the singular values along the diagonal, $V$ is a unitary matrix $n \times n$ and $U$
is an $(n + \nu) \times n$ matrix with $U^{\ast} U= I$, that is, $U$ belongs to the Stiefel
manifold $V_{n,n + \nu}$. If  we let $y_1, \ldots, y_n$ be
the squared singular values of $Y$, then it is known that
\begin{equation} \label{JacobianSVD} 
	dY \propto \left( \prod_{j=1}^n y_j^{\nu} \right) \Delta(y)^2  \, dU dV \,  dy_1 \cdots dy_n \,
	\end{equation}
where $dU$ is the invariant measure on the Stiefel manifold, and $dV$ is the Haar measure on $U(n)$,
see e.g.\ \cite{Edelman_Rao} and \cite{Zheng_Tse}\footnote{
Note that in \cite[page 10]{Edelman_Rao} the Jacobian is given in terms of the singular
values $\sigma_j = \sqrt{y_j}$ with a factor $\prod_{j=1}^n \sigma_j^{2 \nu + 1}$.  Since 
$2 \sigma_j d\sigma_j = d y_j$ we obtain \eqref{JacobianSVD} with factor  $\prod_{j=1}^n y_j^{\nu}$.}.
Combining \eqref{eq: G_to_Y} and \eqref{JacobianSVD} we obtain a probability measure proportional
to 
\begin{equation} \label{pdfUVy} 
	  \left(\prod_{k = 1}^n x_k^{-(n + \nu)} \right)\left( \prod_{j=1}^n y_j^{\nu} \right) 
			\Delta(y)^2 \, e^{-\Tr(V^{\ast} \Sigma^{\ast} \Sigma V (X^{\ast}X)^{-1})} dU dV dy_1 \cdots dy_n.
	\end{equation}

Since we are only interested in the squared singular values of $Y$, we integrate out the $U$ and $V$ part in \eqref{pdfUVy}.	
The integral over $U$ only contributes to the constant. The integration over $V$ is done by 
means of the Harish-Chandra/Itzykson-Zuber integral \cite{Harish-Chandra, Itzykson_Zuber}
\begin{equation} \label{HCIZ}
	\int_{U(n)} e^{-\Tr(V^{\ast} \Sigma^{\ast} \Sigma V (X^{\ast}X)^{-1})} dV = 
	C_n \frac{\det\left[e^{-y_j/x_k}\right]_{j, k = 1}^n}{\Delta(y) \Delta(x^{-1})},
\end{equation}
where $C_n$ is a (known) constant only depending on $n$. From \eqref{pdfUVy} and \eqref{HCIZ}
we obtain that  the density of squared singular values of $Y$ is proportional to
\begin{equation} \label{pdfy} 
	  \left(\prod_{k = 1}^n x_k^{-(n + \nu)} \right) \left( \prod_{j=1}^n y_j^{\nu} \right) 
			\frac{\Delta(y) \, \det \left[ e^{-y_j/x_k} \right]_{j,k=1}^n}{\Delta(x^{-1})}.
	\end{equation}
Using $\Delta(x^{-1}) = (-1)^{n(n-1)/2} \prod_{k=1}^n x_k^{-n+1} \Delta(x)$ and bringing the products
into the determinant, we immediately obtain  \eqref{Y0pdf} with a proportionality
constant that is independent of $x_1, \ldots, x_n$. This proves  the lemma.
\end{proof}
 
We can now prove Theorem \ref{thm: distr_sv_prod}.

\begin{proof}
Suppose that the squared singular values of $X$ have joint probability density function \eqref{Xpdf}.
Then the squared singular values are distinct almost surely, and we obtain from Lemma \ref{lemma: distr_sv_prod_fixed}, after averaging out over $X$, that the squared singular values of $Y=GX$ have a joint probability density function that is proportional to
\begin{equation} \label{pdfy3}
	\Delta(y) \int_0^{\infty} \cdots \int_0^{\infty} \det\left[ \frac{y_j^{\nu}}{x_k^{\nu + 1}} e^{-y_j/x_k}\right]_{j, k = 1}^n 
		\det\left[f_{k-1}(x_j)\right]_{j, k = 1}^n dx_1 \cdots dx_n.
\end{equation}
The multiple integral in \eqref{pdfy3} can be evaluated with the Andreief identity, 
see e.g.\ \cite[Chapter 3]{Deift_Gioev},
\begin{multline*} 
	\int \cdots \int \det\left[ \phi_j(x_k) \right]_{j,k=1}^n \det\left[\psi_k(x_j) \right]_{j,k=1}^n \,  d\mu(x_1) \cdots d\mu(x_n) 	\\
	= n! \det \left[ \int \phi_j(x) \psi_k(x) d\mu(x) \right]_{j,k=1}^n 
	\end{multline*}
	and the result is that \eqref{pdfy3} is proportional to \eqref{Ypdf}
with functions
\begin{align*} g_{k}(y) = \int_{0}^{\infty} \frac{y^{\nu}}{x^{\nu + 1}} e^{-y/x} f_{k}(x) dx 
	 = \int_0^{\infty} x^{\nu} e^{-x} f_{k}\left(\frac{y}{x}\right) \frac{dx}{x}  
	\end{align*}
and this completes the proof of Theorem  \ref{thm: distr_sv_prod}.
\end{proof}

\begin{remark}
We emphasize that in Theorem \ref{thm: distr_sv_prod} we do not assume that the probability
distribution on $X$ is invariant under (left or right) multiplication with unitary
matrices.
\end{remark}

\begin{remark} It is of interest to find other random matrices $A$ so that
multiplication with $A$ preserves the biorthogonal structure of squared singular
values. In a forthcoming work we will show that this is the case for multiplication
with truncated unitary matrices.
The main issue is to find a suitable analogue of the Harish-Chandra/Itzykson
Zuber formula \eqref{HCIZ}. 
\end{remark}

Theorem \ref{thm: distr_sv_prod} has a number of immediate consequences that we list now.

\section{Corollaries}

We can apply Theorem \ref{thm: distr_sv_prod} to any random matrix $X$ for which the squared
singular values have a joint probability density function of the form \eqref{Xpdf}. 

In all the examples below, we will see the appearance of Meijer G-functions. This is actually
quite naturally, because of its connections with the Mellin transform. In particular if
$f_k$ in Theorem \ref{thm: distr_sv_prod} is a Meijer G-function, then also
$g_k$ in \eqref{eq: mellin_transf_f} is a Meijer G-function, see formula
\eqref{eq: MeijerG_Mellin_transform} in the appendix.

\subsection{$X$ is a Ginibre matrix}

Suppose $X = G_1$ is itself a complex Ginibre random matrix of size $(n + \nu_1) \times n$, $\nu_1 \geq 0$. 
Then it is known that the squared singular values of $X$ have a joint p.d.f.\ proportional to
\begin{equation*}
	\Delta(x)^2 \prod_{j=1}^n x_j^{\nu_1} e^{-x_j}
\end{equation*}
which can be written in the form \eqref{Xpdf} with 
\begin{equation} \label{eq: expfunction} 
	f_k(x) = x^{\nu_1 + k} e^{-x} =  \MeijerG{1}{0}{0}{1}{-}{\nu_1 + k}{x}. 
\end{equation}

Assume now that $Y = G_M \cdots G_1$ is the product of $M$ independent complex Ginibre matrices where
$G_k$ has size $(n + \nu_k) \times (n + \nu_{k-1})$ and all $\nu_k \geq 0$ with $\nu_0 = 0$.
Then we can apply Theorem \ref{thm: distr_sv_prod} $M-1$ times and using 
\eqref{eq: MeijerG_convolution1} and \eqref{eq: expfunction} we immediately find:

\begin{corollary} \label{cor: prod_M_gin}
The joint probability density function of the squared singular values of $Y = G_M \cdots G_1$ 
is proportional to
\begin{equation}
\Delta(y) \det\left[w_{k-1}(y_j)\right]_{j, k = 1}^n
\end{equation}
where 
\begin{equation}
	w_k(y) = \MeijerG{M}{0}{0}{M}{-}{\nu_M, \dots, \nu_1 + k}{y}.
\end{equation}
\end{corollary}
This is the result of Akemann, Ipsen and Kieburg \cite{Akemann_Ipsen_Kieburg} mentioned in the introduction.

\subsection{$X$ is the inverse of a product of Ginibre matrices}

A second application of Theorem \ref{thm: distr_sv_prod}  is inspired by the recent 
work of Forrester \cite{Forrester} who considered the product 
\begin{equation} \label{eq: prod_gin_inv}
	Y = G_M \cdots G_1(\tilde{G}_K \cdots \tilde{G}_1)^{-1}
\end{equation} 
of $M$ Ginibre random matrices with the inverse of a product of $K$ Ginibre random matrices.
Here it is assumed that $\tilde{G}_j$ has size $(n + \tilde{\nu}_j) \times (n + \tilde{\nu}_{j-1})$
with all $\tilde{\nu}_j \geq 0$, and $\tilde{\nu}_0 = \tilde{\nu}_K = 0$, so that $\tilde{G}_K \cdots \tilde{G}_1$
is a square matrix.

From Corollary \ref{cor: prod_M_gin} we know that the squared singular values of 
$\tilde{G}_K \cdots \tilde{G}_1$ have a joint probability density function proportional to 
\begin{equation} \label{eq: MeijerGK}
	\Delta(x) \det\left[\phi_{k-1}(x_j)\right]_{j, k = 1}^n, \qquad 
		\phi_{k}(x) = \MeijerG{K}{0}{0}{K}{-}{\tilde{\nu}_K, \dots, \tilde{\nu}_1 + k}{x}.
\end{equation}

The following simple lemma shows that the squared singular values of $(\tilde{G}_K \cdots \tilde{G}_1)^{-1}$
then also have the structure of a polynomial ensemble. 
\begin{lemma} \label{lemma: distr_sv_inv}
Let $X$ be a random matrix of size $n \times n$ such that the squared singular 
values $x_1, \ldots, x_n$ of $X$ have a joint probability density function proportional to
\begin{equation} \label{eq: lemma_distr_sv}
	\Delta(x) \det\left[\phi_k (x_j)\right]_{j, k = 1}^n,
\end{equation}
for certain functions $\phi_k$. 
Then the squared singular values of $Y = X^{-1}$ have a joint probability density function proportional to
\begin{equation} \label{Xinvpdf}
	\Delta(y) \det\left[\psi_k \left(y_j\right)\right]_{j, k = 1}^n
\end{equation}
with
\begin{equation} \label{psik}
	\psi_k(y) = y^{-n-1}\phi_k\left(y^{-1}\right), \qquad k = 1, \ldots, n.
\end{equation}
\end{lemma}
\begin{proof}
The squared singular values of $X^{-1}$ are given by $y_j = x_j^{-1}$, $j =1, \ldots, n$.
making the change of variables $x_j \mapsto y_j = x_j^{-1}$ in \eqref{eq: lemma_distr_sv} gives 
us the joint probability density function of the squared singular values of $X^{-1}$. The Jacobian
of this change of variables is $\ds (-1)^n \prod_{j = 1}^n y_j^{-2}$.
Noting also that 
\[ \Delta(y^{-1}) = (-1)^{n(n-1)/2} \prod_{j = 1}^n y_j^{-n + 1} \Delta(y), \] 
we find that
the joint probability density function of $y_1, \ldots, y_n$ is therefore proportional to
\begin{equation*}
	\prod_{j=1}^n y_j^{-2} \, \prod_{j=1}^n y_j^{-n+1} \, \Delta(y) \, \det\left[\phi_{k-1}(y_j^{-1}) \right]
\end{equation*}
which is indeed \eqref{Xinvpdf} with $\psi_k$ given by \eqref{psik}.
\end{proof}

The class of Meijer G-functions is closed under inversion of the argument and under
multiplication by a power of the independent variable, see \eqref{eq: MeijerG_timesxrho} 
and \eqref{eq: MeijerG_xinverse}. It follows that if $\phi_k$ in 
Lemma~\ref{lemma: distr_sv_inv} is a Meijer G-function, then so is $\psi_k$.
To be precise, if
\[ \phi_k(x) = \MeijerG{m}{n}{p}{q}{a_1, \ldots, a_p}{b_1, \ldots, b_q}{x} \]
then
\[ \psi_k(y) = \MeijerG{n}{m}{q}{p}{-b_1-n, \ldots, -b_q-n}{-a_1-n, \ldots, -a_p-n}{x}. \] 

If we apply this to \eqref{eq: MeijerGK} we see that the squared singular values of
$X = (\tilde{G}_K \cdots \tilde{G}_1)^{-1}$ have a joint p.d.f. of the form \eqref{Xpdf}
with
\[ f_{k}(x) = \MeijerG{0}{K}{K}{0}{-\tilde{\nu}_K-n, \ldots,  -\tilde{\nu}_2 -n,  -\tilde{\nu}_1 - n -k }{-}{x}. \] 

Then a repeated application of Theorem \ref{thm: distr_sv_prod}
and formula \eqref{eq: MeijerG_convolution1} gives the following result of  \cite[Proposition 3]{Forrester}.
\begin{corollary} \label{cor: prod_M_gin_inv_K_gin}
Let $\tilde{G}_1, \cdots \tilde{G}_K$ and $G_1, \ldots, G_M$ be independent complex Ginibre matrices
where $\tilde{G}_j$ has size $(n + \tilde{\nu}_j) \times (n + \tilde{\nu}_{j-1})$ and
$G_j$ has size $(n + \nu_j) \times (n + \nu_{j-1})$ with $\tilde{\nu}_1, \ldots, \tilde{\nu}_{K-1} \geq 0$,
$\nu_1, \ldots, \nu_M \geq 0$ and $\nu_0 = \tilde{\nu}_0 = \tilde{\nu}_K = 0$.
Then the squared singular values  $y_1, \ldots, y_n$ of $Y$ given in \eqref{eq: prod_gin_inv} with $K \geq 1$,
have a joint probability density function  proportional to
\begin{equation} \label{eq: jpdf_gin_invgin}
	\Delta(y) \det\left[w_{k-1}(y_j)\right]_{j, k = 1}^n
\end{equation}
where 
\begin{equation} \label{eq: wk_product_with_inverses}
	w_k(y) = \MeijerG{M}{K}{K}{M}{-\tilde{\nu}_K - n, \ldots, -\tilde{\nu}_2 - n, -\tilde{\nu}_1 - n - k}{\nu_M, \ldots, \nu_1}{y}.
\end{equation}
\end{corollary}

\subsection{$X$ is a truncation of a random unitary matrix}

As a third application we consider a new example, where we start from a matrix $X$ which is
a truncated unitary matrix. 
Let $U$ be an $l \times l$ Haar distributed unitary matrix and let $X$ be 
the $(n + \nu_1) \times n$ upper left block of $U$, where $ \nu_1 \geq 0$ and $l \geq 2n + \nu_1$. 
Then the squared singular values of $X$ are in $(0,1)$ with a joint p.d.f that is
proportional to (see e.g.\  \cite[Proposition 2.1]{Jiang})
\begin{equation} \label{Jacobipdf}
	\prod_{1 \leq j < k \leq n} (x_k-x_j)^2 \, \prod_{j=1}^n x_j^{\nu_1}  (1-x_j)^{l - 2n - \nu_1}, 
		\qquad \text{all } x_j \in (0,1). 
\end{equation}
This is a Jacobi ensemble with parameters $\nu_1$ and $l-2n - \nu_1$. Note that in the case $l < 2n + \nu_1$
the truncation $X$ always has $1$ as a singular value, and \eqref{Jacobipdf} is not valid.
We can rewrite \eqref{Jacobipdf} as \eqref{Xpdf}
with functions
\begin{equation} \label{Jacobifunctions} 
	f_k(x) = \begin{cases} x^{\nu_1+k} (1-x)^{l-2n-\nu_1}  &  \text{ for } 0 < x < 1, \\
		0 & \text{ elsewhere.} \end{cases}
	\end{equation}
Then 
\begin{equation} \label{Jacobi_MeijerG}
	f_k(x) = \Gamma(l - 2n - \nu_1 + 1) \MeijerG{1}{0}{1}{1}{l-2n+k+1}{\nu_1+k}{x}.
	\end{equation}

Let $M \geq 1$ and form the product $Y= G_{M-1} \cdots G_1 X$ where 
$G_j$ is a complex Ginibre matrix of size $(n+\nu_j) \times (n +\nu_{j-1})$ for $j = 1, \ldots, M-1$.
Then Theorem \ref{thm: distr_sv_prod} together with  \eqref{eq: MeijerG_convolution1} 
and \eqref{Jacobi_MeijerG} gives the following. 
\begin{corollary} \label{cor: prod_gin_trunc}
Let $X$ be the $(n + \nu_1) \times n$ truncation of a unitary matrix of size $l \times l$ with $l \geq 2n + \nu_1$.
Let $M \geq 1$ and let $Y = G_{M-1} \cdots G_1 X$ where each  $G_j$ is a complex Ginibre matrix of size $(n+\nu_j) \times (n+\nu_{j-1})$
with $\nu_1, \ldots, \nu_M \geq 0$.
Then the squared singular values $y_1, \ldots, y_n$ of $Y$ have a joint probability density function proportional to
\begin{equation} \label{eq: jpdf_gin_unit}
	\Delta(y) \det\left[w_{k-1}(y_j)\right]_{j, k = 1}^n
\end{equation}
where
\begin{equation} \label{eq: wk_truncation}
	w_k(y) = \MeijerG{M}{0}{1}{M}{l - 2n + 1 + k}{\nu_M, \ldots, \nu_2, \nu_1 + k}{y}, \qquad y > 0.
\end{equation}
\end{corollary}


\section{Integral Representations and Hard Edge Scaling Limit}

\subsection{Kernels $K_{\nu_1, \ldots, \nu_M}$}

For any set of functions $w_0, \ldots, w_{n-1}$, the 
probability density function \eqref{eq: jpdf_gin}  is a polynomial ensemble
and we already noted in the introduction that the correlation kernel takes the form 
\begin{equation} \label{eq: corr_kernel_general}
	K_n(x, y) = \sum_{k = 0}^{n-1} P_k(x) Q_k(y)
\end{equation}
where, for each $k = 0, \ldots, n-1$, $P_k$ is a polynomial of degree $k$ and $Q_k$ is in the span of $w_0, \ldots, w_{k}$ such that
\begin{equation} \label{eq: biorth}
	\int_0^{\infty} P_j(x) Q_k(x) dx = \delta_{j, k}.
\end{equation}
For the case of weight functions \eqref{wk1} that are associated with the product
of $M$ Ginibre matrices, it was shown in \cite{Akemann_Ipsen_Kieburg} and \cite{Kuijlaars_Zhang}
that the functions $P_j$ and $Q_k$ have contour integral representations,
which was used in \cite{Kuijlaars_Zhang} to derive a double integral representation
of the correlation kernel \eqref{eq: corr_kernel_general}.
Based on this double integral representation the following scaling limit was obtained in
\cite[Theorem 5.3]{Kuijlaars_Zhang}.

\begin{theorem}  \label{thm: scaling_KZ}
Let $M \geq 1$ and $\nu_1, \ldots, \nu_M \geq 0$ be fixed integers. Then the kernels $K_n$ have
the scaling limit
\[ \lim_{n \to \infty} \frac{1}{n} K_n \left( \frac{x}{n}, \frac{y}{n} \right) =
	K_{\nu_1, \ldots, \nu_M}(x,y). \] The limiting kernel has a double
	integral represention 
\begin{equation} \label{eq: lim_kernel_hard_edge}
	K_{\nu_1, \ldots, \nu_M}(x, y) = \frac{1}{(2 \pi i)^2} \int_{-1/2 - i \infty}^{-1/2 + i \infty} ds 
	\oint_{\Sigma} dt \prod_{i = 0}^M 
	\frac{\Gamma(s + \nu_i + 1)}{\Gamma(t + \nu_i + 1)} \frac{\sin(\pi s)}{\sin(\pi t)} \frac{x^t y^{-s-1}}{s-t}
\end{equation}
where $\Sigma$ is a contour in $\{ t \in \mathbb{C} \mid \Re t > - 1/2 \}$ encircling the positive real axis, 
as illustrated in Figure \ref{fig: plot_contours_int_k_nu_m}. 
\end{theorem}

The kernels \eqref{eq: lim_kernel_hard_edge} have the alternative representation in terms
of Meijer G-functions
\begin{multline} \label{eq: Knu_integral}
	K_{\nu_1, \ldots, \nu_M}(x, y) \\
	= \int_0^1 \MeijerG{1}{0}{0}{M+1}{-}{-\nu_0, -\nu_1, \ldots, -\nu_M}{ux} 
	\MeijerG{M}{0}{0}{M+1}{-}{\nu_1, \ldots, \nu_M,\nu_0}{uy} du,
\end{multline}
with $\nu_0 = 0$, see also \cite{Kuijlaars_Zhang}.

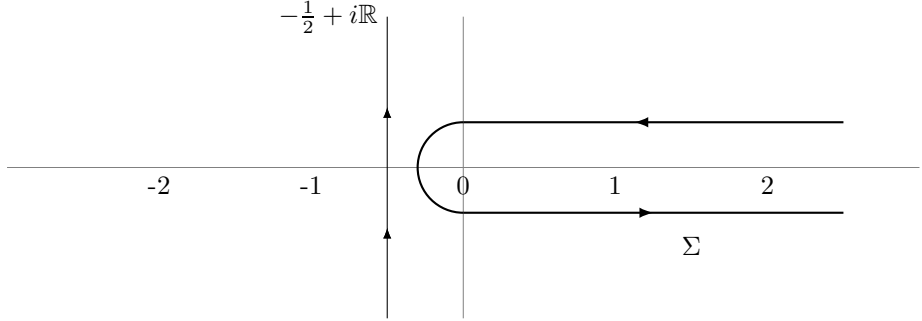
\begin{figure}[t!]
\centering
\begin{tikzpicture}[scale=2]
\tikzset{deco/.style 2 args={
            decoration={             
                        markings,   
                        mark=at position {#1} with { 
                                    \arrow{latex},
                                    \node[anchor=\pgfdecoratedangle-90] {#2};
                        }
            },
            postaction={decorate}
    }
}

\draw[help lines] (-3,0) -- (3,0);
\draw[deco={0.3}{}, deco={0.7}{}] (-0.5, -1) -- (-0.5, 1) coordinate (yaxis);

\draw[help lines] (0,-1) -- (0,1);

\path[draw,line width=0.8pt] (0, 0.3) arc (90:270:0.3);
\draw[deco={0.55}{},line width=0.8pt] (2.5, 0.3) -- (0, 0.3); 
\draw[deco={0.5}{},line width=0.8pt] (0, -0.3) -- (2.5, -0.3); 

\node[below] at (-2, 0) {-2};
\node[below] at (-1, 0) {-1};
\node[below] at (0, 0) {0};
\node[below] at (1, 0) {1};
\node[below] at (2, 0) {2};
\node[below] at (1.5, -0.4) {$\Sigma$};
\node[left] at (yaxis) {$-\frac{1}{2} + i \mathbb{R}$};
\end{tikzpicture}
\caption{The contours $-\frac{1}{2} + i \mathbb R$ and $\Sigma$ in  
the double integral \eqref{eq: lim_kernel_hard_edge}}
\label{fig: plot_contours_int_k_nu_m}
\end{figure}

In this section, we consider the polynomial ensemble  \eqref{eq: jpdf_gin_unit} from Corollary \ref{cor: prod_gin_trunc}
that is associated with the product of complex Ginibre random matrices with one truncated unitary matrix.
Following \cite{Kuijlaars_Zhang} and \cite{Forrester} we are able to obtain 
integral representations for $P_k$ and $Q_k$ in this case as well, and from this a double integral representation 
for the correlation kernel.  While keeping $\nu_1, \ldots, \nu_M$ fixed and letting $l$ grow at least 
as $2n$, we obtain the limiting kernel $K_{\nu_1, \ldots, \nu_M}$ at the hard edge also in this case.


\subsection{Integral representations for $Q_k$ and $P_k$} \label{subsec: int_rep}

So in the rest of this section, we assume that we work with the functions $w_k$ in \eqref{eq: wk_truncation}.
The corresponding polynomials $P_k$  are such that $P_k$ is monic of degree $k$ with 
\begin{equation} \label{eq: Pk_character}
	\int_0^{\infty} P_k(x) w_j(x) dx = 0 \qquad \text{for } j=0, \ldots, k-1 
	\end{equation}
and the functions $Q_k$ satisfy 
\begin{equation} \label{eq: Qk_character} 
	\int_0^{\infty} x^j Q_k(x) dx = \delta_{j,k} \qquad \text{for } j=0, \ldots, k, 
	\end{equation}
with $Q_k$ in the linear span of $w_0, \ldots, w_k$.
The polynomials $P_k$ and functions $Q_k$ have the following integral representation.

\begin{proposition}
We have
\begin{equation} \label{eq: int_rep_qkx}
	Q_k(x) = \frac{1}{2\pi i} \int_{c - i\infty}^{c + i\infty} q_k(s) 
		\frac{\prod_{j = 1}^M \Gamma(s + \nu_j)}{\Gamma(s + l - 2n + 1)} x^{-s} ds
\end{equation}
where $c > 0$ and
\begin{equation} \label{eq: qk}
q_k(s) = \frac{\Gamma(l - 2n + 2k + 2)}{\prod_{j = 0}^M \Gamma(k + 1 + \nu_j)} \frac{(s-k)_k}{(s + l - 2n + 1)_k}.
\end{equation}
\end{proposition}
Recall that $\nu_0 = 0$ and that the Pochhammer symbol is given by
\begin{equation*}
	(a)_k = a(a+1) \cdots (a+k-1) = \frac{\Gamma(a+k)}{\Gamma(a)}.
\end{equation*}
\begin{proof}
The functions $w_k$ from \eqref{eq: wk_truncation} have the integral representation
\begin{equation} \label{eq: wk_integral} 
	w_k(x) = \frac{1}{2\pi i} \int_{c-i \infty}^{c + i \infty} 
	\frac{ (s+\nu_1)_k}{(s+l - 2n + 1)_k}  \frac{ \prod_{j=1}^M \Gamma(s+\nu_j)}{\Gamma(s+l-2n+1)} x^{-s} ds, \qquad x > 0.
	\end{equation}
Then it is easy to see that the linear span of $w_0, \ldots, w_k$ consists of all functions as in the right-hand side
of \eqref{eq: int_rep_qkx} with $q_k$ being a rational function in $s$ such that
$(s+l-2n + 1)_k q_k(s)$ is a polynomial of degree $\leq k$. Since \eqref{eq: qk} is of that type, we see that
$Q_k$ belongs to the linear span of $w_0, \ldots, w_k$.

By the properties of the Mellin transform, we have from \eqref{eq: int_rep_qkx} that
\begin{equation} \label{eq: qk_Mellin} 
	\int_0^{\infty} x^{s-1} Q_k(x) dx = q_k(s) \frac{\prod_{j = 1}^M \Gamma(s + \nu_j)}{\Gamma(s + l - 2n + 1)}, \qquad s > 0. 
	\end{equation}
Since $q_k$ has zeros in $1, \ldots, k$ we find from \eqref{eq: qk_Mellin} that 
\[ \int_0^{\infty} x^j Q_k(x) dx = 0, \qquad \text{for } j = 0, \ldots, k-1. \]
The prefactor in \eqref{eq: qk} has been chosen so that
\[ \int_0^{\infty} x^k Q_k(x) dx = 1 \]
as can be readily verified from \eqref{eq: qk} and \eqref{eq: qk_Mellin}. Thus \eqref{eq: Qk_character} holds
and the proposition is proved.
\end{proof}

Notice that we can rewrite $Q_k$ as a Meijer G-function
\begin{equation*}
Q_k(x) = \frac{\Gamma(l - 2n + 2k + 2)}{\prod_{i = 0}^M \Gamma(k + 1 + \nu_i)} \MeijerG{M+1}{0}{2}{M+1}{-k, l - 2n + k + 1}{\nu_0, \nu_1, \ldots, \nu_M}{x}.
\end{equation*}
There is a similar integral representation for $P_k$.

\begin{proposition}
We have
\begin{equation} \label{eq: int_rep_pnx}
P_k(x) = \frac{\prod_{j= 0}^M \Gamma(k + 1 + \nu_j)}{\Gamma(l - 2n + 2k + 1)} 
	\frac{1}{2\pi i} \oint_{\Sigma_k} \Gamma(t-k) \frac{\Gamma(t + l - 2n + k + 1)}{\prod_{j = 0}^M \Gamma(t + 1 + \nu_j)} x^{t} dt
\end{equation}
where $\Sigma_k$ is a closed contour encircling  the interval $[0, k]$ once in the positive direction 
and such that $\Re t > -1$ for $t \in \Sigma_k$.
\end{proposition}
\begin{proof}
The integrand in the right-hand side of \eqref{eq: int_rep_pnx} is a meromorphic function with simple poles $0, 1, \ldots, k$ 
inside the contour. Since $\Re t > -1$ for $t \in \Sigma_k$, we have that the other poles are outside.  
Hence, by the residue theorem we have that the right-hand side of \eqref{eq: int_rep_pnx}
defines a polynomial of degree at most $k$, and in fact
\begin{align} \nonumber
	P_k(x) & = \frac{\prod_{j = 0}^M \Gamma(k + 1 + \nu_j)}{\Gamma(l - 2n + 2k + 1)} \sum_{t = 0}^k 
		\Res_{t}\left(\Gamma(t-k) \frac{\Gamma(t + l - 2n + k + 1)}{\prod_{j = 0}^M \Gamma(t + 1 + \nu_j)}\right) x^t \\
	& = \label{eq: sum_rep_pnx}
			\sum_{t = 0}^k \frac{(-1)^{k - t}}{(k - t)!} \frac{\Gamma(l - 2n + k + t + 1)}{\Gamma(l - 2n + 2k + 1)}   
			\prod_{j = 0}^M \frac{\Gamma(k + 1 + \nu_j)}{\Gamma(t + 1 + \nu_j)} x^t.
\end{align}
The leading coefficient is $1$ and so $P_k$ is indeed  a monic polynomial of degree $k$.

To check the orthogonality condition \eqref{eq: Pk_character}, we recall that by \eqref{eq: wk_integral} and the
properties of the Mellin transform
\begin{equation*}
	\int_0^{\infty} x^{s-1} w_j(x) dx = \frac{ (s+\nu_1)_j \prod_{i=1}^M \Gamma(s+\nu_i)}{(s+l - 2n + 1)_j \Gamma(s+l-2n+1)}.
\end{equation*}
 And so, if we use \eqref{eq: int_rep_pnx} and interchange the integrals
\begin{multline} \label{eq: Pkwj_integral}
\int_0^{\infty} P_k(x) w_j(x) dx  \\
	=  \frac{c_k}{2\pi i} 
	\oint_{\Sigma_k} \Gamma(t-k) \frac{\Gamma(t + l - 2n + k + 1)}{\prod_{i = 0}^M \Gamma(t + 1 + \nu_i)} 
	\frac{(t+1+\nu_1)_j}{(t + l - 2n + 2)_j} \frac{\prod_{i=1}^M \Gamma(t+1+\nu_i)}{\Gamma(t + l - 2n + 2)} dt
	\end{multline}
with $c_k = \ds \frac{\prod_{i = 0}^M \Gamma(k + 1 + \nu_i)}{\Gamma(l - 2n + 2k + 1)}$.
In case $j = 0, \ldots, k-1$, the integrand in \eqref{eq: Pkwj_integral} simplifies to
\[   \frac{(t+1+\nu_1)_j  (t+ l - 2n + j + 2)_{k-j-1}}{(t-k)_{k+1}} \]
which is a rational function in $t$ with poles at $t=0, 1, \ldots, k$ only, and these are inside the contour $\Sigma_k$.
In addition it is $O(t^{-2})$ as $t \to \infty$.
Thus by moving the contour $\Sigma_k$ to infinity, we see that \eqref{eq: Pkwj_integral} vanishes
for $j=0, 1, \ldots, k-1$, and we obtain \eqref{eq: Pk_character}.
\end{proof}

Formula \eqref{eq: sum_rep_pnx} shows that $P_k$ is a hypergeometric polynomial, namely
\begin{equation*}
P_k(x) = (-1)^k \prod_{i=1}^M \frac{\Gamma(k+1+\nu_i)}{\Gamma(\nu_i + 1)} \frac{\Gamma(l - 2n + k + 1)}{\Gamma(l - 2n + 2k + 1)} \HyperG{2}{M}{-k, l - 2n + k + 1}{1+\nu_1, \ldots, 1+\nu_M}{x}.
\end{equation*}
Finally, we notice that $P_k$ can also be identified as a Meijer G-function
\begin{equation*}
P_k(x) = - \frac{\prod_{i = 0}^M \Gamma(k + 1 + \nu_i)}{\Gamma(l - 2n + 2k + 1)} \MeijerG{0}{2}{2}{M+1}{k + 1, -(l - 2n + k)}{-\nu_0, \ldots, -\nu_M}{x}.
\end{equation*}


\subsection{Hard edge limit} \label{subsec: hard edge}

We proceed to obtain a double integral representation for the kernel $K_n$.  

\begin{proposition}
We have
\begin{multline} \label{eq: int_rep_knxy}
 K_n(x, y) = \frac{1}{(2\pi i)^2} \int_{-1/2 - i\infty}^{-1/2 + i\infty} ds 
	\oint_{\Sigma} dt \prod_{j = 0}^M \frac{\Gamma(s + 1 + \nu_j)}{\Gamma(t+1+\nu_j)} \\
	\frac{\Gamma(t+1-n) \Gamma(t + l - n + 1)}{\Gamma(s+1-n) \Gamma(s + l - n + 1)} \frac{x^t y^{-s-1}}{s-t}
\end{multline}
where $\Sigma$ is a closed contour encircling $0, 1, \ldots, n$ once in the positive direction 
such that $\Re t > -1/2$ for $t \in \Sigma$.
\end{proposition}
\begin{proof}
The correlation kernel \eqref{eq: corr_kernel_general} can be written as
\begin{multline*}
K_n(x, y) = \frac{1}{(2\pi i)^2} \int_{c - i\infty}^{c + i\infty} ds 
	\oint_{\Sigma} dt \prod_{j = 0}^M \frac{\Gamma(s + \nu_j)}{\Gamma(t+1+\nu_j)} \\
	\sum_{k = 0}^{n-1} (l - 2n + 2k + 1) \frac{\Gamma(t-k)}{\Gamma(s-k)}\frac{\Gamma(t + l - 2n + k + 1)}{\Gamma(s + l - 2n + k + 1)} x^t y^{-s}
\end{multline*}
where we used the representations \eqref{eq: int_rep_pnx} and \eqref{eq: int_rep_qkx} for $P_k(x)$ 
and $Q_k(y)$. By using the functional equation $\Gamma(z+1) = z \Gamma(z)$, one can see that
\begin{multline*}
(s-t-1) (l - 2n + 2k + 1) \frac{\Gamma(t-k)}{\Gamma(s-k)}\frac{\Gamma(t + l - 2n + k + 1)}{\Gamma(s + l - 2n + k + 1)} \\
 = \frac{\Gamma(t-k)}{\Gamma(s-k-1)}\frac{\Gamma(t + l - 2n + k + 2)}{\Gamma(s + l - 2n + k + 1)} - \frac{\Gamma(t-k+1)}{\Gamma(s-k)}\frac{\Gamma(t + l - 2n + k + 1)}{\Gamma(s + l - 2n + k)}
\end{multline*}
which means that we have a telescoping sum 
\begin{multline} \label{eq: telescope}
(s-t-1) \sum_{k = 0}^{n-1}(l - 2n + 2k + 1) \frac{\Gamma(t-k)}{\Gamma(s-k)}\frac{\Gamma(t + l - 2n + k + 1)}{\Gamma(s + l - 2n + k + 1)}  \\
	= \frac{\Gamma(t-n+1)}{\Gamma(s-n)}\frac{\Gamma(t + l - n + 1)}{\Gamma(s + l - n)} - \frac{\Gamma(t+1)}{\Gamma(s)}
	\frac{\Gamma(t + l - 2n + 1)}{\Gamma(s + l - 2n)}.
\end{multline}
By taking $c = 1/2$ and letting $\Sigma$ encircle $0, 1, \ldots, n$ such that $\Re(t) > -1/2$ for $t \in \Sigma$, we ensure that $s-t-1 \neq 0$ whenever $s \in c + i\R$ and $t \in \Sigma$. And so we obtain that
\begin{multline*}
K_n(x, y) = \\
\frac{1}{(2\pi i)^2} \int_{1/2 - i\infty}^{1/2 + i\infty} ds \oint_{\Sigma} dt 
	\prod_{j = 0}^M \frac{\Gamma(s + \nu_j)}{\Gamma(t+1+\nu_j)}\frac{\Gamma(t-n+1)}{\Gamma(s-n)}\frac{\Gamma(t + l - n + 1)}{\Gamma(s + l - n)} \frac{x^t y^{-s}}{s-t-1} \\
- \frac{1}{(2\pi i)^2} \int_{1/2 - i\infty}^{1/2 + i\infty} ds \oint_{\Sigma} dt \prod_{j = 0}^M \frac{\Gamma(s + \nu_j)}{\Gamma(t+1+\nu_j)}\frac{\Gamma(t+1)}{\Gamma(s)}\frac{\Gamma(t + l - 2n + 1)}{\Gamma(s + l - 2n)} \frac{x^t y^{-s}}{s-t-1}.
\end{multline*}
The integrand of the second double integral has no singularities inside $\Sigma$ and hence 
the $t$-integral vanishes by Cauchy's theorem. By finally making the 
change of variable $s \mapsto s+1$ in the first double integral, we obtain \eqref{eq: int_rep_knxy}.
\end{proof}

\begin{remark}
The proofs in subsections \ref{subsec: int_rep} and \ref{subsec: hard edge}
are modelled after those in  \cite{Kuijlaars_Zhang}, but there are slight differences in
all proofs. 

We want to emphasize one difference which has to do with the telescoping sum \eqref{eq: telescope}.
The left-hand side of \eqref{eq: telescope} has the factors $l - 2n +k +1$ which come
from the product of the prefactors in \eqref{eq: qk} and \eqref{eq: int_rep_pnx}. The
corresponding prefactors in \cite[formulas (3.2) and (3.8)]{Kuijlaars_Zhang} are each other
inverses, and as a consequence there is no such factor. However it is remarkable that 
the factors $l-2n+k+1$ are actually necessary for the telescoping sum \eqref{eq: telescope} 
to hold.
\end{remark}

We notice that we can rewrite the kernel in terms of Meijer G-functions :

\begin{corollary}
We have
\begin{equation} \label{eq: meijer_rep_knxy}
K_n(x, y) = \int_0^1 \MeijerG{0}{2}{2}{M+1}{n, -(l - n)}{-\nu_0, \ldots, -\nu_M}{ux} 
\MeijerG{M+1}{0}{2}{M+1}{-n, l - n}{\nu_0, \ldots, \nu_M}{uy} du.
\end{equation}
\end{corollary}
\begin{proof}
Since
\begin{equation*}
\frac{x^t y^{-s-1}}{s-t} = -\int_0^1 (ux)^t (uy)^{-s-1} du,
\end{equation*}
the kernel \eqref{eq: int_rep_knxy} can be rewritten as
\begin{align*}
K_n(x, y) = &-\int_0^1 \left(\frac{1}{2\pi i} \oint_{\Sigma} 
	\frac{\Gamma(t+1-n) \Gamma(t + l - n + 1)}{\prod_{j = 0}^M \Gamma(t+1+\nu_j)} (ux)^t dt\right) \\
&\times \left(\frac{1}{2\pi i} \int_{-1/2 - i\infty}^{-1/2 + i\infty}  
	\frac{\prod_{j = 0}^M \Gamma(s + 1 + \nu_j)}{\Gamma(s+1-n) \Gamma(s + l - n + 1)} (uy)^{-s-1} ds \right) du.
\end{align*}
By the definition of a Meijer G-function and making the change of variables $t \mapsto -t$ and $s \mapsto s-1$, we obtain the identity 
\eqref{eq: meijer_rep_knxy}.
\end{proof} 

Using the integral representation for $K_n$, we can derive the scaling limit at the hard edge.
\begin{theorem}
With $\nu_1, \ldots, \nu_M$ being fixed and with $l$ growing at least as $2n$, we have
\begin{equation} \label{eq: lim_knxy}
\lim_{n \to \infty} \frac{1}{(l - n)n} K_n\left(\frac{x}{(l - n)n}, \frac{y}{(l - n)n}\right) = K_{\nu_1, \ldots, \nu_M}(x, y).
\end{equation}
uniformly for $x, y$ in compact subsets of the real positive axis, where
\begin{multline} \label{eq: int_rep_k_nu_m}
	K_{\nu_1, \ldots, \nu_M}(x, y) \\
	= \frac{1}{(2\pi i)^2} \int_{-1/2 - i\infty}^{-1/2 + i\infty} ds \oint_{\Sigma} dt 
	\prod_{j = 0}^M \frac{\Gamma(s + 1 + \nu_j)}{\Gamma(t+1+\nu_j)} \frac{\sin(\pi s)}{\sin(\pi t)} \frac{x^t y^{-s-1}}{s-t}.
\end{multline}
The contour $\Sigma$ starts at $+\infty$ in the upper half plane and returns to $+\infty$ in 
the lower half plane encircling the positive real axis such that $\Re t > -1/2$ for $t \in \Sigma$ 
(see also Figure \ref{fig: plot_contours_int_k_nu_m} on
page \pageref{fig: plot_contours_int_k_nu_m}).
\end{theorem}
\begin{proof}
By using identity \eqref{eq: int_rep_knxy}, we know
\begin{multline*}
\frac{1}{(l - n)n} K_n \left(\frac{x}{(l - n)n}, \frac{y}{(l - n)n}\right) = 
\frac{1}{(2\pi i)^2} \int_{-1/2 - i\infty}^{-1/2 + i\infty} ds 
\oint_{\Sigma} dt \prod_{j = 0}^M \frac{\Gamma(s + 1 + \nu_j)}{\Gamma(t+1+\nu_j)} \\
 \frac{\Gamma(t+1-n) \Gamma(t + l - n + 1)}{\Gamma(s+1-n) \Gamma(s + l - n + 1)} \frac{x^t y^{-s-1}}{s-t}(l - n)^{s-t}n^{s-t}.
\end{multline*}
Euler's reflection formula tells us that
\begin{equation*}
\Gamma(z)\Gamma(1-z) = \frac{\pi}{\sin(\pi z)},
\end{equation*}
and so we see that
\begin{equation*}
\frac{\Gamma(t-n+1)}{\Gamma(s-n+1)} = \frac{\Gamma(n-s)}{\Gamma(n-t)} \frac{\sin(\pi s)}{\sin(\pi t)}.
\end{equation*}
Furthermore, as $n \to \infty$ we also know \cite[formula 5.11.13]{Olver:2010}
\begin{equation*}
\frac{\Gamma(n-s)}{\Gamma(n-t)} = n^{t-s} \left(1 + O(n^{-1})\right)
\end{equation*}
and similarly
\begin{equation*}
\frac{\Gamma(t + l - n + 1)}{\Gamma(s + l - n + 1)} = (l - n)^{t-s} \left(1 + O((l - n)^{-1})\right).
\end{equation*}
Hence, if we deform the contour $\Sigma$ to a two sided, unbounded contour as in Figure \ref{fig: plot_contours_int_k_nu_m} 
and apply the identities above, we immediately obtain  identity \eqref{eq: lim_knxy}, provided that we can take 
the limit inside the integral. This can be justified by using the dominated convergence theorem 
(see \cite[Theorem 5.3]{Kuijlaars_Zhang} for details).
\end{proof}

\section{Borodin Biorthogonal Ensembles} \label{sec: Bor_biorth_ens}

In this final section we consider the biorthogonal ensembles \eqref{eq:Bor1}
that were studied by Borodin in \cite{Borodin}, see also \cite{Claeys_Romano,Muttalib}. 
These are  determinantal point process on $[0,\infty)$, whose correlation kernels are expected to have
interesting scaling limits at the hard edge $x = 0$. This was proved in \cite{Borodin} for the cases where $w$ 
is either a special Jacobi weight
\begin{equation} \label{Bor2} 
	w(x) = \begin{cases} x^{\alpha}, &  0 < x \leq 1, \\
	0, & x > 1, \end{cases} \qquad \alpha > -1, 
	\end{equation}
or a Laguerre weight
\begin{equation} \label{Bor3} 
	w(x) = x^{\alpha} e^{-x} \qquad x > 0, \qquad \alpha > -1.
	\end{equation}
In both cases it was shown that a scaling limit at the origin leads to the following correlation kernel
that depends on $\alpha$ and $\theta$,
\begin{equation} \label{Bor4} 
	K^{(\alpha, \theta)}(x,y) = 
	\theta x^{\alpha} \int_0^1 J_{\frac{\alpha+1}{\theta}, \frac{1}{\theta}} (xu)
		J_{\alpha+1, \theta} ((y u)^{\theta}) u^{\alpha} du. 
		\end{equation}
where $J_{a,b}$ is Wright's generalization of the Bessel function given by
\begin{equation} \label{Bor5} 
	J_{a,b}(x) = \sum_{j=0}^{\infty} \frac{(-x)^j}{j! \Gamma(a+j b)}. 
	\end{equation}
The kernels \eqref{Bor4} are related to the Meijer G-kernel $K_{\nu_1, \ldots, \nu_M}$ 
in case $\theta$ or $1/\theta$ is an integer. This is our final result.

\begin{theorem} \label{thm:theta-ensemble} Let $M \geq 1$ be an integer. 
\begin{enumerate}
\item[\rm (a)] Then we have 
\begin{equation} \label{Bor6} 	
	M^M K^{(\alpha,\frac{1}{M})}(M^M x, M^M y)  = \\
	 \left(\frac{x}{y} \right)^{\alpha}  K_{\nu_1, \ldots, \nu_M}(x,y) 
	\end{equation}
	with parameters 
\begin{equation} \label{Bor7}  
	\nu_j = \alpha + \frac{j-1}{M}, \qquad j=1, \ldots, M.
	\end{equation}
\item[\rm (b)]	
	We also have
\begin{equation} \label{Bor8} 
	x^{\frac{1}{M}-1} K^{(\alpha,M)}(M x^{\frac{1}{M}}, M y^{\frac{1}{M}}) = 
		K_{\tilde{\nu}_1, \ldots, ,\tilde{\nu}_M}(y,x) 
		\end{equation}
	with parameters
\begin{equation} \label{Bor9}  
	\tilde{\nu}_j = \frac{\alpha}{M} - 1 + \frac{j}{M}, \qquad j=1, \ldots, M.
	\end{equation}
	\end{enumerate}
\end{theorem}
The parameters \eqref{Bor7} and \eqref{Bor9} come in an arithmetic progression with step size $1/M$, and
therefore they cannot all be integers if $M \geq 2$. This is in contrast to the limiting kernels
obtained from the products of random matrices where the $\nu_j$ are necessarily integers.

\begin{proof}
If $b$ is a rational number, then \eqref{Bor5} can be expressed as a Meijer G-function, 
see \cite[formula (22)]{GLM99} and \cite[formula (13)]{GLM2000}. 
For the case when $b=M$ with $M \geq 2$ an integer, we have
\begin{align} \nonumber
	J_{a, M}(x) & = (2\pi)^{\frac{M-1}{2}} M^{-a + 1/2} \frac{1}{2\pi i} 
	\int_{c-i\infty}^{c+i\infty} \frac{\Gamma(s)}{\prod_{j=0}^{M-1}  \Gamma\left(\frac{a}{M}-s+\frac{j}{M}\right)}  \left( \frac{x}{M^M}\right)^{-s} ds  \\
	& = \label{Bor6_2}  (2\pi)^{\frac{M-1}{2}} M^{-a + 1/2} \MeijerG{1}{0}{0}{M+1}{-}{0,  -\frac{a}{M}+\frac{1}{M}, -\frac{a}{M}+\frac{2}{M}, \ldots, -\frac{a}{M} + 1 }{\frac{x}{M^M}}
\end{align}
and for $b=1/M$,
\begin{align} \nonumber 
	J_{a, \frac{1}{M}}(x) & = (2\pi)^{- \frac{M-1}{2}} M^{1/2} \frac{1}{2\pi i}
	\oint_{\Sigma} \frac{\prod_{k=0}^{M-1} \Gamma\left(t+ \frac{k}{M}\right)}{\Gamma(a-t)} \left( \frac{x^M}{M^M} \right)^{-t} dt \\
	\label{Bor7_2} 
& =  (2\pi)^{- \frac{M-1}{2}} M^{1/2} \MeijerG{M}{0}{0}{M+1}{-}{0, \frac{1}{M}, \ldots, \frac{M-1}{M}, 1-a}{\frac{x^M}{M^M}}
\end{align}
where $\Sigma$ is a contour encircling the negative real axis. 

Inserting \eqref{Bor6_2} and \eqref{Bor7_2} into \eqref{Bor4} we obtain for $\theta = 1/M$ with a positive integer $M$, 
\begin{multline} 
K^{(\alpha,\frac{1}{M})}(x, y) \\ 
 = M^{-(\alpha + 1)M} x^{\alpha} \int_0^1 
	\MeijerG{1}{0}{0}{M+1}{-}{0,  - \alpha, - \alpha - \frac{1}{M}, \ldots, -\alpha - \frac{M-1}{M}}{\frac{ux}{M^M}}  \\
  \times \MeijerG{M}{0}{0}{M+1}{-}{0, \frac{1}{M}, \ldots, \frac{M-1}{M}, -\alpha}{\frac{uy}{M^M}} u^{\alpha} du, 
	\label{Bor8_2}
\end{multline}
and after a rescaling of variables $x \mapsto M^M x$, $y \mapsto M^M y$,
\begin{multline}
	M^M K^{(\alpha,\frac{1}{M})}(M^M x, M^M y)  \\
	= x^{\alpha} \int_0^1 \MeijerG{1}{0}{0}{M+1}{-}{0,  - \alpha, - \alpha - \frac{1}{M}, \ldots, -\alpha - \frac{M-1}{M}}{ux}  \\
		\times \MeijerG{M}{0}{0}{M+1}{-}{0, \frac{1}{M}, \ldots, \frac{M-1}{M}, -\alpha}{uy} u^{\alpha} du. \label{Bor9_2}
\end{multline}
Then we realize that for a Meijer G-function $G(z)$, we have that $z^{\alpha} G(z)$ is again a 
Meijer G-function with parameters shifted by $\alpha$, see formula \eqref{eq: MeijerG_timesxrho}. Thus by \eqref{Bor9_2}
\begin{multline}
M^M K^{(\alpha,\frac{1}{M})}(M^M x, M^M y)   \nonumber \\
 = \left(\frac{x}{y}\right)^{\alpha} \int_0^1 
	\MeijerG{1}{0}{0}{M+1}{-}{0,  - \alpha, - \alpha - \frac{1}{M}, \ldots, -\alpha - \frac{M-1}{M}}{ux} \nonumber \\
	\times \MeijerG{M}{0}{0}{M+1}{-}{\alpha, \alpha + \frac{1}{M}, \ldots, \alpha + \frac{M-1}{M}, 0}{uy} du \label{Bor10}
\end{multline}
which proves part (a) of the theorem because of \eqref{eq: Knu_integral}.
\medskip

Part (b) follows in a similar way. Alternatively it can be obtained from part (a)
because of the formula
\[ \frac{1}{\theta} x^{\frac{1}{\theta} - 1} 
	K^{(\alpha, \theta)}(x^{\frac{1}{\theta}}, y^{\frac{1}{\theta}}) 
	= \left(\frac{x}{y} \right)^{\alpha'} K^{(\alpha', \frac{1}{\theta})}(y,x), 
		\qquad \alpha' = \frac{\alpha+1}{\theta} - 1 \]
	which can be easily deduced from \eqref{Bor4}.
\end{proof}

\section*{Acknowledgements}

We thank Peter Forrester for useful discussions and providing us with a 
copy of \cite{Forrester_Liu}.

The authors are supported by KU Leuven Research Grant OT/12/073 and the Belgian Interuniversity
Attraction Pole P07/18. The first author is also supported by FWO Flanders projects G.0641.11 and G.0934.13, and by Grant
No. MTM2011-28952-C02 of the Spanish Ministry of Science and Innovation.


\appendix
\section{The Meijer G-function} \label{appendix: meijer_g}

For ease of reference, we collect in this appendix the definition and properties of the Meijer G-function that are
used in this paper. By definition, the Meijer G-function is given by the following contour integral:
\begin{equation} \label{eq: def_meijer_g}
	\MeijerG{m}{n}{p}{q}{a_1, \ldots, a_p}{b_1, \ldots, b_q}{z} = 
	\frac{1}{2\pi i} \int_{L} \frac{\prod_{j = 1}^m \Gamma(s + b_j) \prod_{j = 1}^n \Gamma(1 - a_j - s)}
	{\prod_{j = m+1}^q \Gamma(1 - b_j - s) \prod_{j = n + 1}^p \Gamma(s + a_j)} z^{-s} ds,
\end{equation}
where the branch cut of $z^{-s}$ is taken along the negative $x$-axis. Furthermore, it is also assumed that
\begin{itemize}
\item $m, n, p, q$ are integers such that $0 \leq p \leq n$ and $0 \leq q \leq m$;
\item the real (or complex) parameters $a_1, \ldots, a_p$ and $b_1, \ldots, b_q$ satisfy the conditions
\begin{equation*}
	a_k - b_j \neq 1, 2, 3, \ldots \text{ for } k = 1, \ldots, n \text{ and } j = 1, \ldots, m.
\end{equation*}
I.e., none of the poles of $\Gamma(b_j + u)$ coincides with any of the poles of $\Gamma(1 - a_k - u)$. 
\end{itemize} 
The contour $L$ is such that all the poles of $\Gamma(u + b_j)$ are on the left of the path 
while the poles of $\Gamma(1 - a_j + u)$ are on the right of the path. In typical situations
the contour is a vertical line $c + i \mathbb R$ with $c > 0$.

The Mellin transform of an integrable function $w$ on $[0, \infty)$ is
\[ (\mathcal M w)(s) = \int_0^{\infty} x^{s-1} w(x) dx, \qquad a < \Re s < b. \]
The inverse Mellin transform is
\[ w(x) = \frac{1}{2\pi i} \int_{c-i\infty}^{c+i\infty} (\mathcal M w)(s) x^{-s} ds, \]
where $a < c < b$. Thus for a Meijer G-function which is defined and
integrable on the positive half-line we have 
\begin{equation} \label{eq: MeijerG_Mellin_transform} 
	\int_0^{\infty} x^{s-1} \MeijerG{m}{n}{p}{q}{a_1, \ldots, a_p}{b_1, \ldots, b_q}{x} dx 
		= \frac{\prod_{j = 1}^m \Gamma(s + b_j) \prod_{j = 1}^n \Gamma(1 - a_j - s)}
		{\prod_{j = m+1}^q \Gamma(1 - b_j - s) \prod_{j = n + 1}^p \Gamma(s + a_j)}. 
		\end{equation}
		
The Mellin convolution of two Meijer G-functions is again a Meijer G-function. A special
case of this is
\begin{equation} \label{eq: MeijerG_convolution1}
		\int_{0}^{\infty} x^{\nu-1}  e^{-x} 
			\MeijerG{m}{n}{p}{q}{a_1, \ldots, a_p}{b_1, \ldots, b_q}{ \frac{y}{x}} dx 
			= \MeijerG{m+1}{n}{p}{q+1}{a_1, \ldots, a_p}{\nu, b_1, \ldots, b_q}{y},
\end{equation}
provided that the integral in the left-hand side converges.

Further identities are 
\begin{equation} \label{eq: MeijerG_timesxrho}
	x^{\alpha} \MeijerG{m}{n}{p}{q}{a_1, \ldots, a_p}{b_1, \ldots, b_q}{x} = 
	\MeijerG{m}{n}{p}{q}{a_1 + \alpha, \ldots, a_p + \alpha}{b_1 + \alpha, \ldots, b_q + \alpha}{x}
\end{equation}
and
\begin{equation} \label{eq: MeijerG_xinverse}
	\MeijerG{m}{n}{p}{q}{a_1, \ldots, a_p}{b_1, \ldots, b_q}{x^{-1}} = 
	\MeijerG{n}{m}{q}{p}{1-b_1, \ldots, 1-b_q}{1-a_1, \ldots, 1-a_p}{x}.
	\end{equation}
For more details, we refer the reader to \cite{Beals_Szmigielski, Luke}.


\end{document}